\documentclass[11pt,reqno]{amsproc}
\usepackage{amssymb}
\usepackage{amscd}
\usepackage[utf8x]{inputenc}
\usepackage{geometry}
\usepackage{graphicx}
\usepackage{subfigure}
\usepackage{setspace}
\usepackage{mathptmx}
\usepackage[T1]{fontenc} 
\newtheorem{theorem}{Theorem}[section]

\newtheorem{definition}{Definition}[section]

\newcommand{\RR}{\mathbb{R}}

\newcommand{\QQ}{\mathbb{Q}}

\newcommand{\ARD}{\textnormal{ARD}}

\newcommand{\continuum}{\mathfrak{c}}

\newcommand{\mc}{\mathcal}
\newcommand{\mf}{\mathfrak}
\newcommand{\bez}{\backslash}

\begin{document}
\vspace*{1.2cm}
\hspace*{7.5 cm}				
\begin{minipage}[t]{5cm}
\begin{footnotesize}

\begin{flushright}
\textit{Key words:} 
\end{flushright}
%
% Keywords
%
\textit{outer measure, rational distances, Bernstein set, Vitali set, nonmeasurable set}
\end{footnotesize}
\end{minipage}

\vspace*{0.5 cm}
\setlength{\parindent}{0 cm}
%
% Authors:
%
Author: Marcin Michalski
\vspace*{0.7 cm}
{\fontsize{15pt}{11pt}
\footnote{The author was supported by grant S40012/K1102 from the Faculty
of Fundamental Problems of Technology of Wrocław University of Technology.}
\begin{center}
   \textbf{\\A note on sets avoiding rational distances
%(title of the paper 13 points, bold, line spacing 
%at least 15 points)
}\footnote{This work is a conference paper accepted to $13^{th}$ Students' Science Conference (2015), Polanica Zdrój, Poland.} 
\end{center}}
\vspace*{0.5 cm}
\hspace{0.5cm}
\begin{center}
\begin{minipage}[t]{13cm}
\setlength{\parindent}{0.5cm}
\begin{footnotesize}

%
% Abstract
%

\indent\textsc{Abstract} In this paper we shall give a short proof of the result originally obtained by Ashutosh Kumar that for each $A\subset \mathbb{R}$ there exists $B\subset A$ full in $A$ such that no distance between two distinct points from $B$ is rational. We will construct a Bernstein subset of $\RR$ which also avoids rational distances. We will show some cases in which the former result may be extended to subsets of $\mathbb{R}^2$, i. e. it remains true for measurable subsets of the plane and if $non(\mathcal{N})=cof(\mathcal{N})$ then for a given set of positive outer measure we may find its full subset which is a partial bijection and avoids rational distances.

\end{footnotesize}
\end{minipage}
\end{center}
\bigskip
\bigskip

%
% First section
%

%\vspace*{0.1cm}
\setlength{\parindent}{0.5cm}
\section{   \fontsize{11pt}{11pt} \MakeUppercase{Introduction and preliminaries}}
\vspace*{0.3cm}
We will use a standard set-theoretic notation based on [3]. We denote the real line by $\mathbb{R}$ and the set of rationals by $\mathbb{Q}$. By $\alpha, \beta, \gamma, ...$ we denote ordinal numbers, with an exception of $\lambda$ ($\lambda^*$ respectively) which will stand for Lebesgue measure (outer Lebesgue measure respectively). By $\omega$ we denote the set of natural numbers and by $\mathfrak{c}=2^{\omega}=|\mathbb{R}|$ we denote the continuum - the cardinality of $\mathbb{R}$. We will say that a set $A$ is countable if $|A|\leq \omega$. \\ We call a family of sets $\mathcal{I}$ an ideal of sets if $\mathcal{I}$ is closed under finite unions and taking subsets. We call it a $\sigma$ -ideal of sets if it is an ideal closed under countable unions. For $\sigma$-ideal $\mathcal{I}$ and the space $X$ we define following cardinal coefficients from the Cichoń's diagram:
\begin{eqnarray*}
	non(\mathcal{I})&=&\min\{|A|: A\subset X \land A\notin\mathcal{I}\}
	\\
	cof(\mathcal{I})&=&\min\{|\mathcal{A}|: \mathcal{A}\subset\mathcal{I} \land (\forall I\in\mathcal{I})(\exists A\in\mathcal{A})(I\subset A) \}	
\end{eqnarray*}
We shall consider the classical examples of $\sigma$-ideals: $\mathcal{M}$ - the family of meager sets and $\mathcal{N}$ - the family of null sets. We denote Borel subsets of the considered space by $\mathcal{B}$. 
\begin{definition}
We say thay a set $A$ is
\begin{itemize}
\item $\mc{I}-positive$ Borel set if it is Borel and does not belong to $\mc{I}$,
\item $\mathcal{I}$-nonmeasurable if $A\notin\sigma(\mathcal{B}\cup \mathcal{I})$,
\item completely $\mathcal{I}$-nonmeasurable, if it intersects each $\mc{I}-positive$ Borel set, but also does not contain any such set.
\end{itemize}

\end{definition}
The classic example of completely nonmeasurable set for any reasonable $\sigma$-ideal (nontrivial, containing points, possessing Borel base) is a Bernstein set - a set that intersects every perfect set, but does not contain any such set. For $\mc{I}=\mc{N}$ sets which are $\mathcal{I}$-measurable are Lebesgue measurable and for $\mc{I}=\mc{M}$ sets which are $\mathcal{I}$-measurable posses the Baire property.
\\
For $A\subset\RR^2$ and $x, y\in\RR$ we define a horizontal slice of $A$: $A^y=\{x\in\RR: (x, y)\in A\}$ and a vertical slice of $A$: $A_x=\{y\in\RR: (x, y)\in A\}$. 
\\
Let us consider a well known equivalence relation $\sim$ over $\mathbb{R}$: $x\sim y \leftrightarrow x-y\in\mathbb{Q}$. Let $V$ be a selector of the partition of $\mathbb{R}$ induced by $\sim$, that is $|V\cap [x]_{\sim}|=1$ for each $x\in\mathbb{R}$. A set with such a property we call a Vitali set. Let us proceed to main definitions of the paper.
\begin{definition}
We say that a set $A$ avoids rational distances (shortly: $A$ is an $\ARD$ set) if for each $x,y\in A, x\neq y$, the distance between $x$ and $y$ is irrational.
\end{definition}
\begin{definition}
We say that a set $A\subset B$ is full (in $B$) if for each $X$ of positive measure we have $\lambda^*(X\cap B)=\lambda^*(X\cap A)$.
\end{definition}

It is clear that every Vitali set is an $\ARD$ set. Observe also that for a set $B$ of finite outer measure $A\subset B$ is full in $B$ iff $A$ has the same outer measure as $B$. 
\\
Let us recall the following result of Gitik and Shelah (see [2]).
\begin{theorem}(Gitik, Shelah)\label{GS}
Let $(A_n)_{n\in\omega}$ be a sequence of subsets of $\RR^n$. Then there exists a sequence $(B_n)_{n\in\omega}$ such that for every $n\in\omega$ we have $B_n\subset A_n$ and $\lambda^*(B_n)=\lambda^*(A_n)$

\end{theorem}

\section{   \fontsize{11pt}{11pt} \MakeUppercase{On the real line}}
\vspace*{0.3cm}

In 2012 Ashutosh Kumar (see [5]) proved that for each $A\subset \mathbb{R}$ there exists an ARD set $B\subset A$ full in $A$. We shall give a short proof of this result.
\begin{theorem}(Kumar)\label{Kumar}
Let $A\subset \RR$ be a set of positive outer measure. Then there exists an $\ARD$ set $B\subset A$ full in $A$.

\end{theorem}
\begin{proof}
Let $V$ be a Vitali set and let us enumerate rationals $\QQ=\{q_n: n\in\omega\}$. For each $n\in\omega$ let:
$$
A_n=\{v\in V: v+q_n\in A\}
$$
By Theorem \ref{GS} let us take a sequence $(B_n)_{n\in\omega}$ such that for each $n$ $B_n\subset A_n$ and $\lambda^*(B_n)=\lambda^*(A_n)$. Since $\bigcup_{n\in\omega}(A_n+q_n)=A$ it is easy to verify that $B=\bigcup_{n\in\omega}(B_n+q_n)$ is the set.

\end{proof}

It is a nontrivial exercise to prove that there exists a Vitali set of full outer measure. We will construct a Vitali set with a bit stronger property.

\begin{theorem}
There exists a Vitali set that is also a Bernstein set.

\end{theorem}

\begin{proof}
Let $\{P_\alpha: \alpha<\mf{c}\}$ be an enumeration of all perfect subsets of $\RR$ and $\{C_\alpha; \alpha<\mf{c}\}$ be an enumeration of all equivalance classes of relation $\sim$. We will construct desired set by transfinite induction.
\\
At the first step we choose arbitrarily $p_0\in P_0$ then $c_0\in C_0$ if $p_0\notin C_0$ (otherwise we set $c_0=p_0$) and $e_0\in P_0$, $e_0\neq p_0, c_0$.
\\
At the step $\xi<\continuum$ let us assume that we have transfinite sequences $(p_\alpha: \alpha<\xi)$, $(c_\alpha: \alpha<\xi)$, $(e_\alpha: \alpha<\xi)$ such that:
\begin{enumerate}
\item $p_\alpha\in P_\alpha$, $c_\alpha\in C_\alpha$, $e_\alpha\in P_\alpha$ for all $\alpha<\xi$,
\item $\{p_\alpha: \alpha<\xi\}\cap \{e_\alpha: \alpha<\xi\}=\emptyset$,
\item $\{c_\alpha: \alpha<\xi\}\cap \{e_\alpha: \alpha<\xi\}=\emptyset$,
\item $|(\{p_\alpha: \alpha<\xi\}\cup \{c_\alpha: \alpha<\xi\})\cap C_\beta|\leq 1$ for each $\beta<\continuum$,
\item $|\{e_\alpha: \alpha<\xi\}\cap C_\beta|\leq 1$ for each $\beta<\continuum$.
\end{enumerate}
We will extend our sequences in such a way that above properties will be preserved. Let us consider the following set:
$$
E_\xi=\{x\in C_\beta: \beta<\continuum, (\{p_\alpha: \alpha<\xi\}\cup \{c_\alpha: \alpha<\xi\})\cap C_\beta\neq\emptyset \}\cup\{e_\alpha: \alpha<\xi\}.
$$
Its cardinality is at most $|\xi|<\mf{c}$ since each class $C_\beta$ is countable and our sequences have a length $\xi$. Every perfect set has the cardinality of $\mf{c}$ so the set $P_\xi\bez E_\xi$ is nonempty. Let us choose then $p_\xi\in P_\xi\bez E_\xi$. If $p_\alpha\in C_\xi$ for some $\alpha\leq\xi$, then set $c_\xi=p_\alpha$, otherwise we choose arbitrarily $c_\xi\in C_\xi\bez \{e_\alpha: \alpha<\xi\}$ (the latter is nonempty since intersection of $C_\xi$ and $\{e_\alpha: \alpha<\xi\}$ is at most one point). Eventually let us consider the following set:
$$
E'_\xi=\{p_\alpha, c_\alpha: \alpha\leq\xi\}\cup\{e_\alpha: \alpha<\xi\}\cup\{x\in C_\beta: \beta<\continuum, \{e_\alpha: \alpha<\xi\}\cap C_\beta\neq\emptyset\}.
$$
Similarly to the previous reasoning it has the cardinality of $|\xi|$, so we may pick $e_\xi\in P_\xi\bez E'_\xi$.
\\
This finishes the construction and $V=\{p_\alpha: \alpha<\mf{c}\}\cup\{c_\alpha: \alpha<\mf{c}\}$ is the set.

\end{proof}

\section{   \fontsize{11pt}{11pt} \MakeUppercase{On the real plane}}
\vspace*{0.3cm}

Since being in rational distance on the plane is not an equivalence relation, the idea behind the Theorem \ref{Kumar} cannot be directly utilized to prove a similar result for the plane and higher dimensions. However, we are not totally helpless and there are some cases in which we are able to give some answers. First, we shall consider the case of measurable sets. Before we proceed let us denote by $d$ a standard Euclidean metric on $\RR^2$, by $B(x, r)$ an open ball with origin $x$ and radius $r$ and let $S(x, r)=\delta B(x, r)$ (the frame of ball $B$ - in $\RR^2$ it is a circle).

\begin{theorem}\label{mierzalne}
Let $A$ be a measurable subset of $\RR^2$ of positive measure. Then there exists an $\ARD$ set $B\subset A$ full in $A$. 

\end{theorem}
\begin{proof}
Since $A$ is measurable, there exists $F_\sigma$ set $F\subset A$ such that $\lambda(A)=\lambda(F)$. Let $\{B_\alpha: \alpha<\mf{c}\}$ be an enumeration of all Borel subsets of $F$ of positive measure. We shall construct the desired set inductively.
\\
Let us choose point $b_0\in B_0$. Now, let assume that we have already chosen a transfinite sequence $(b_\alpha: \alpha<\xi)$ for $\xi<\continuum$ such that for each $\alpha<\xi$ we have $b_\alpha\in B_\alpha$ and for all $\alpha, \beta<\xi$ we have $d(b_\alpha, b_\beta)\in \QQ$. Let us pick
$$
b_\xi\in B_\xi\bez\big(\bigcup_{\alpha<\xi}\bigcup_{q\in\QQ_+}S(b_\alpha, q)\big).
$$
Such a choice can be made since by the Fubini theorem for each set of positive measure there is a positive set $A_\alpha\subset\RR$ such that for every $x\in A_\alpha$ the set ${B_\alpha}_x$ has the cardinality of $\continuum$ and its intersection with $\bigcup_{\alpha<\xi}\bigcup_{q\in\QQ_+}S(b_\alpha, q)$ has a cardinality of $|\xi|<\continuum$.
\\
This finishes the construction and $B=\{b_\alpha: \alpha<\continuum\}$ is the set.

\end{proof}

It is quite easy to verify that the above theorem holds also in the case of $\RR^n$ for any natural $n$ and the proof remains almost intact.

Peter Komjath (see [4]) proved that the real plane can be colored by countably many colors such a way that each color is an ARD set. It implies that for each set of positive outer measure there exists its ARD subset of positive measure too. With some additional assumptions we may prove a little more surprising result. Let us recall that $f\subset A\subset\RR^2$ is a partial bijection, if there exists a set $A_1\subset \RR$ such that $f: A_1\rightarrow \RR$ is injection and $\{(x, f(x)): x\in A_1\}\subset A$.

\begin{theorem}
Assume that $non(\mathcal{N})=cof(\mathcal{N})$. Let $A\subset \RR^2$ be a set of positive outer measure. Then there exists a partial bijection $f$ that is an $\ARD$ full subset of $A$.

\end{theorem}
\begin{proof}
Let us assume that $non(\mathcal{N})=cof(\mathcal{N})=\kappa$. To construct a full subset $B$ of $A$ we only have to make sure that $B$ intersects $X\cap A$ for each $X$ of positive measure such that $\lambda^*(X\cap A)>0$. Since $cof(\mathcal{N})=\kappa$ and thanks to result of Cichoń, Kamburelis and Pawlikowski (see [1]) we have to consider only $\kappa$ sets of positive measure. Let $\{F_\alpha: \alpha<\kappa\}$ be an enumeration of "testing" sets of positive measure such that their intersections with $A$ have positive outer measure. We shall construct desired function $f$ via the transfinite induction.
\\
$A\cap F_0$ has a positive outer measure so there exists a set $A_0\subset \RR$ of positive outer measure such that for each $x\in A_0$ we have that $(A\cap F_0)_x$ also has a positive outer measure (otherwise $A\cap F_0$ would be null). Let us choose $x_0\in A_0$ and $y_0\in (A\cap F_0)_{x_0}$.
\\
Next, assume that we have a transfinite sequence of pairs $((x_\alpha, y_\alpha))_{\alpha<\xi}$ for $\xi<\kappa$ such that for each $\alpha<\xi$ we have $(x_\alpha, y_\alpha)\in A\cap F_\alpha$ and $\{(x_\alpha, y_\alpha): \alpha<\xi\}$ is an ARD set and a partial bijection. Let us see if we can extend our sequence and preserve these properties. Again, $A\cap F_\xi$ has a positive outer measure so there exists a set $A_\xi\subset \RR$ of positive outer measure such that for each $x\in A_\xi$ we have that $(A\cap F_\xi)_x$ also has a positive outer measure. Furthermore, since $non(\mc{N})=\kappa$ and $\xi<\kappa$, $A_\xi\bez \{x_\alpha: \alpha<\xi\}$ also has a positive outer measure and so has $\Big(A\cap F_\xi\bez \big(\bigcup_{\alpha<\xi}\bigcup_{q\in\QQ_+}S((x_\alpha, y_\alpha), q)\big)\Big)_x$ for each $x\in A_\xi\bez \{x_\alpha: \alpha<\xi\}$. It means, that we may pick $x_\xi\in A_\xi\bez \{x_\alpha: \alpha<\xi\}$, $y_\xi\in \Big(A\cap F_\xi\bez \big(\bigcup_{\alpha<\xi}\bigcup_{q\in\QQ_+}S((x_\alpha, y_\alpha), q)\big)\Big)_{x_\xi}$ and $(x_\xi, y_\xi)$ extends our sequence in the desired way.
\\
This completes the construction and $f=\{(x_\alpha. y_\alpha): \alpha<\kappa\}$ is the set.

\end{proof}

As a concluding remarks let us notice that due to the category analogue of the Theorem \ref{GS} and a general similarities between $\mc{N}$ and $\mc{M}$ we may prove analogue results for the case of category. The notion used in Theorems \ref{Kumar} and \ref{mierzalne} "$A\subset B$ is full in $B$" should be replaced by "$A$ is nonmeager in B" which means that whenever $X\cap B$ is not meager, then also $X\cap A$ is not meager for every set $X$ possessing the Baire property. The rest of the results holds by replacing $\mc{N}$ with $\mc{M}$.
\\
%
% Second section
%

%%%%%%%%%%%%%%%%%%%%%%%%%%%%%%%%%%%%%%%%%%%%%%%%%%%%%%%%%%%%%%%%%%%%%%%%%%%%%%%%%%%%%%%%%%%%%%%

\begin{center}
{\large\bf References}
\end{center}

[1] J. Cichoń, A. Kamburelis, J.Pawlikowski, On dense subsets of the measure algebra, Proceedings of the American Mathematical Society, Vol. 84, No 1, 1985, pp. 142-146

[2] M. Gitik and S. Shelah, More on simple forcing notions and forcings with ideals, Annals of
Pure and Applied Logic, Vol. 59, 1993, pp. 219-238.

[3] T. Jech, Set Theory, millennium edition, Springer Monographs in Mathematics, Springer-Verlag, 2003.

[4] P. Komjath, A coloring result for the plane, Journal of Applied Analysis, Vol. 5, 1999, pp. 113–117.

[5] A. Kumar, Avoiding rational distances, Real Analysis Exchange, Vol. 38(2), 2012/2013, pp. 493-498.

\end{document}